\documentclass[twoside,11pt]{amsart}
\usepackage{amsmath,amssymb,amsthm}
\usepackage{hyperref}

\begin{document}

\setlength{\textwidth}{145mm} \setlength{\textheight}{203mm}
\setlength{\parindent}{0mm} \setlength{\parskip}{2pt plus 2pt}

\numberwithin{equation}{section}
\newtheorem{theorem}{Theorem}[section]
\newtheorem{corollary}{Corollary}[section]
\newtheorem{proposition}{Proposition}[section]
\newtheorem{lemma}{Lemma}[section]
\newtheorem{definition}{Definition}[section]
\newcommand{\tr}{{\rm tr}}
\newcommand{\dv}{{\rm div}}
\newfont{\got}{eufm9 scaled 1095}
\def\X{\mbox{\got X}}
\def\D{\mbox{\got D}}
\def\L{\mbox{\got L}}
\def\S{\mbox{\got S}}
\def\g{\mbox{\got g}}
\newfont{\w}{msbm9 scaled\magstep1}
\def\R{\mbox{\w R}}
\def\N{\mbox{\w N}}
\newcommand{\norm}[1]{\left\Vert#1\right\Vert}
\newcommand{\const}{\textrm{const.}}
\newcommand{\arctg}{\textrm{arctg}}
\frenchspacing

\newcommand{\KN}{\mathbin{\bigcirc\mspace{-15mu}\wedge\mspace{3mu}}}

\makeatletter
\newcommand*\owedge{\mathpalette\@owedge\relax}
\newcommand*\@owedge[1]{%
  \mathbin{%
    \ooalign{%
      $#1\m@th\bigcirc$\cr
      \hidewidth$#1\m@th\wedge$\hidewidth\cr
    }%
  }%
} \makeatother


\title[Conjugate connections on almost Norden manifolds]{Conjugate connections and statistical structures on almost Norden\\ manifolds}

\author[M. Teofilova]{Marta Teofilova}

\maketitle

{\small

\textsc{Abstract.} Relations between conjugate connections with
respect to the pair of Norden metrics and to the almost complex
structure on almost Norden manifolds are studied. Conjugate
connections of the Levi-Civita connections induced by the Norden
metrics are obtained. Statistical structures on almost Norden
manifolds are considered.
}


\section*{Introduction}

The concept of conjugate connections relative to a metric tensor
field was originally introduced by A. P. Norden in the context of
Weyl geometry \cite{AN}. Such linear connections were independently
developed by H.~Nagaoka and S.~Amari \cite{NA} under the name dual
connections and used by S.~Lauritzen in the definition of
statistical manifolds \cite{SL}. For more details on conjugate
connections and their application to information theory, statistics
and other fields see \cite{Amari}, \cite{SA}, \cite{OC}, \cite{HM},
\cite{KN}, \cite{US}.

Another kind of conjugate connections are those which are dual with
respect to an invertible (1,1)-tensor field \cite{Al}, \cite{Be}.
Conjugate connections relative to an almost complex structure are
studied by A.~M.~Blaga and M.~Crasmareanu in \cite{B}. Relations
between conjugate connections with respect to a symplectic structure
and to a complex structure on K\"ahler manifolds are investigated in
\cite{B1}. Statistical structures and relations between conjugate
connections on Hermitian manifolds are studied in \cite{F},
\cite{Noda}.

The main purpose of the present work\footnote{This work is partially
supported by project FP17--FMI--008 of the Scientific Research Fund,
University of Plovdiv Paisii Hilendarski, Bulgaria.} is to study
relations between both aforementioned types of conjugate connections
on almost complex manifolds with Norden metric (B-metric). For the
sake of brevity, such manifolds will be called almost Norden
manifolds. These manifolds were introduced by A.~P.~Norden \cite{No}
and their geometry was studied for the first time by K.~Gribachev,
D.~Mekerov and G.~Djelepov \cite{Gri} who termed them generalized
B-manifolds.

Since on such manifolds, there exists a pair of Norden metrics, we
can consider conjugate connections with respect to each of these
metrical tensors and their relations to conjugate connections
relative to the almost complex structure. Another aim of this work
is to construct and study statistical structures on almost Norden
manifolds.

The paper is organized as follows. In Section \ref{s1}, we give some
basic information about almost Norden manifolds and conjugate
connections. In Section \ref{s2}, we study the coincidence of
conjugate connections with respect to the Norden metrics and the
almost complex structure. The case of symmetric connections and
completely symmetric connections is also investigated. In Section
\ref{s3}, we study curvature properties of the conjugate connections
of the Levi-Civita connections induced by the pair of Norden
metrics. In Section \ref{s4}, we consider statistical structures on
almost Norden manifolds by constructing families of linear
connections with completely symmetric difference tensor and studying
their curvature properties.

\section{Preliminaries}\label{s1}

\subsection{Almost Norden manifolds}

The triple $(M,J,g)$ is called an \emph{almost Norden manifold}
(almost complex manifold with Norden metric) if $M$ is a
differentiable $2n$-dimensional manifold, $J$ is an almost complex
structure, and $g$ is a pseudo Riemannian metric compatible with $J$
such that
\begin{equation}\label{1}
J^2 X = -X,\qquad g(JX,JY) = - g(X,Y).
\end{equation}
Here and further $X,Y,Z,W$ will stand for arbitrary vector fields on
$M$, i.e. elements in the Lie algebra $\mathfrak{X}(M)$, or vectors
in the tangent space $T_pM$ at an arbitrary point $p\in M$.

Equalities (\ref{1}) imply $g(JX,Y) = g(X,JY)$ which means that the
tensor $\widetilde{g}$ defined by
\begin{equation}
\widetilde{g}(X,Y) = g(X,JY)
\end{equation}
is symmetric and is known as the \emph{associated (twin) metric} of
$g$ ($g$ and $\widetilde{g}$ are called a pair of twin metrics).
This tensor also satisfies the Norden metric property, i.e.
$\widetilde{g}(JX,JY) = -\widetilde{g}(X,Y)$, i.e.
$(M,J,\widetilde{g})$ is also an almost Norden manifold. Both
metrics, $g$ and $\widetilde{g}$, are necessarily of neutral
signature $(n,n)$.

Let us denote by $\nabla^0$ and $\widetilde{\nabla}^0$ the
Levi-Civita connections of $g$ and $\widetilde{g}$, respectively.
The tensor field $F$ defined by
\begin{equation}\label{F-def}
F(X,Y,Z) = (\nabla^0_X \widetilde{g})(Y,Z) = g\left((\nabla^0_X J)Y,
Z\right)
\end{equation}
plays an important role in the geometry of almost Norden manifolds.
It has the following properties
\begin{equation}\label{F-prop}
F(X,Y,Z) = F(X,Z,Y) = F(X,JY,JZ).
\end{equation}

Let $\{e_i\}$ $(i=1,2,...,2n)$ be an arbitrary basis of $T_pM$, and
$g^{ij}$ be the components of the inverse matrix of $g$ with respect
to this basis. The Lie 1-form associated with $F$ and its
corresponding vector $\Omega$ are given by
\begin{equation}\label{theta}
\theta (X) = g^{ij} F(e_i,e_j, X),\qquad \theta(X)=g(X,\Omega).
\end{equation}

A classification of the almost Norden manifolds with respect to the
properties of $F$ is obtained by G.~Ganchev and A.~Borisov in
\cite{Ga}. This classification consists of eight classes: three
basic classes $\mathcal{W}_i$ ($i=1,2,3$), their pairwise direct
sums $\mathcal{W}_i\oplus \mathcal{W}_j$, the widest class
$\mathcal{W}_1\oplus\mathcal{W}_2\oplus\mathcal{W}_3$ and the class
$\mathcal{W}_0$ of the K\"ahler Norden manifolds defined by $F=0$
(i.e. $\nabla^0 J =0$) which is contained in the intersection of
each two classes. The basic classes are distinguished by the
following characteristic conditions, respectively
\begin{equation}\label{Fcl}
\begin{array}{l}
\mathcal{W}_1: F(X,Y,Z) = \frac{1}{2n}\left\{g(X,Y)\theta(Z) +
g(X,JY)\theta(JZ) \right.\smallskip\\ \left.\phantom{\mathcal{W}_1:
F(X,Y,Z)= \frac{1}{2n}} + g(X,Z)\theta(Y) +
g(X,JZ)\theta(JY)\right\};\medskip\\
\mathcal{W}_2: F(X,Y,JZ) + F(Y,Z,JX) + F(Z,X,JY) = 0,\quad \theta =
0;\medskip\\
\mathcal{W}_3: F(X,Y,Z) + F(Y,Z,X) + F(Z,X,Y) = 0.
\end{array}
\end{equation}
The class $\mathcal{W}_1\oplus\mathcal{W}_2$ of the Norden manifolds
(complex manifolds with Norden metric) is the widest integrable
class (i.e. with a vanishing Nijenhuis tensor) and is characterized
also by the condition
\begin{equation*}
F(X,Y,JZ) + F(Y,Z,JX) + F(Z,X,JY) = 0.
\end{equation*}

Let $R^0$ be the curvature tensor of $\nabla^0$, i.e.
\begin{equation*}
R^0(X,Y)Z = \nabla^0_X \nabla^0_Y Z - \nabla^0_Y \nabla^0_X Z -
\nabla^0_{[X,Y]}Z.
\end{equation*}
Its corresponding (0,4)-tensor with respect to $g$ is defined by\\
$R^0(X,Y,Z,W) = g(R^0(X,Y)Z,W)$ and has the following properties
\begin{equation}\label{R0}
\begin{array}{l}
R^0(X,Y,Z,W) = - R^0(Y,X,Z,W) = - R^0(X,Y,W,Z), \smallskip\\
R^0(X,Y,Z,W) + R^0(Y,Z,X,W) + R^0(Z,X,Y,W) = 0.
\end{array}
\end{equation}

Any tensor of type (0,4) which satisfies all three conditions in
(\ref{R0}) is called a \emph{curvature-like tensor}. Then, the Ricci
tensor $\rho(L)$ and the scalar curvature $\tau(L)$ of $L$ are
obtained by
\begin{equation}\label{tau}
\rho(L) (X,Y) = g^{ij}L(e_i,X,Y,e_j),\qquad \tau(L) =
g^{ij}\rho(L)(e_i,e_j).
\end{equation}

A curvature tensor $L$ is called a \emph{K\"ahler tensor} if
$L(X,Y)JZ=JL(X,Y)Z$. Then, for the corresponding (0,4)-type tensor
with respect to $g$, i.e.\\ $L(X,Y,Z,W)=g(L(X,Y)Z,W)$ we have
$L(X,Y,JZ,JW)=-L(X,Y,Z,W)$.

Let $S$ be a tensor of type (0,2), and denote by $\widetilde{S}(X,Y)
= S(X,JY)$. Consider the following (0,4)-tensors:
\begin{equation}\label{psi}
\begin{array}{c}
\psi_1 (S) = g \owedge S, \qquad \psi_2 (S) = \widetilde{g} \owedge
\widetilde{S}, \smallskip\\ \pi_1 = \frac{1}{2}\psi_1(g),\quad \pi_2
= \frac{1}{2}\psi_2(g),\quad \pi_3 =
-\psi_1(\widetilde{g})=\psi_2(\widetilde{g}),
\end{array}
\end{equation}
where $\owedge$ is the Kulkarni-Nomizu product of two (0,2)-tensors,
e.g.\\ $(g\owedge S)(X,Y,Z,W)\hspace{-0.02in} = \hspace{-0.02in}
g(Y,Z)S(X,W) - g(X,Z)S(Y,W) + g(X,W)S(Y,Z) - g(Y,W)S(X,Z)$. The
tensor $\psi_1(S)$ is curvature-like iff $S$ is symmetric, and
$\psi_2(S)$ is curvature-like iff $S$ is symmetric and hybrid with
respect to $J$, i.e. $S(X,Y)=S(Y,X)=-S(JX,JY)$.

On a pseudo-Riemannian manifold $M$ ($\dim M=2n \geq 4$) the Weyl
tensor of a curvature-like tensor $L$ is given by
\begin{equation*}\label{Weyl}
W(L)=L-\frac{1}{2(n-1)}\big
\{\psi_{1}(\rho(L))-\frac{\tau(L)}{2n-1}\pi_{1}\big \}.
\end{equation*}

The square norm of $\nabla^0J$ is defined by
\begin{equation}\label{norm}
||\nabla^0
J||^{2}=g^{ij}g^{kl}g\big((\nabla^0_{e_{i}}J)e_{k},(\nabla^0_{e_{j}}J)e_{l}
\big).
\end{equation}
An almost Norden manifold is called \emph{isotropic K\"ahlerian} if
$||\nabla^0 J||^2=0$.

\subsection{Conjugate connections with respect to a metric tensor and statistical manifolds}

Let $(M,g)$ be a pseudo Riemannian manifold, and $\nabla$ be an
arbitrary linear connection on $M$. Then the linear connection
$\nabla^\ast$ defined by
\begin{equation}\label{g-conj}
Xg(Y,Z) = g(\nabla_X Y, Z) + g(Y,\nabla^{\ast}_X Z),
\end{equation}
is called the \emph{conjugate} (\emph{dual}) \emph{connection of}
$\nabla$ \emph{with respect to} $g$. From (\ref{g-conj}) it is easy
to see that ($\nabla^{\ast})^{\ast} = \nabla$. Hence, $\nabla$ and
$\nabla^\ast$ are said to be mutually conjugate. Also, from
(\ref{g-conj}) it follows that a connection $\nabla$ is
self-conjugate, i.e. $\nabla=\nabla^\ast$ if and only if it is a
\emph{metric} ($g$-\emph{compatible}) \emph{connection}, i.e.
$\nabla g= 0$.

The average connection $\overline{\nabla} = \frac{1}{2}(\nabla +
\nabla^\ast)$ of two mutually conjugate connections is a metric
connection.

Let $R$ and $R^\ast$ be the curvature tensors of $\nabla$ and
$\nabla^\ast$, respectively, and $P$ be the average curvature tensor
of $R$ and $R^\ast$, i.e.
\begin{equation}\label{K}
P(X,Y)Z = \frac{1}{2}\left\{R(X,Y)Z + R^\ast(X,Y)Z\right\}.
\end{equation}
Then, because of the relation $g(R(X,Y)Z,W) = -g(R^\ast(X,Y)W,Z)$,
the corresponding (0,4)-type tensor of $P$ is curvature-like.

Let $\nabla$ be a torsion free (symmetric) connection. Then, it is
known that its conjugate connection $\nabla^\ast$ is also torsion
free if and only if the tensor $\nabla g$ is completely symmetric,
i.e.
\begin{equation}\label{stat}
(\nabla_X g)(Y,Z) = (\nabla_Y g)(X,Z).
\end{equation}
Then the same is valid for $\nabla^\ast g$, i.e. $(\nabla, g)$ and
$(\nabla^\ast, g)$ are both Codazzi pairs. Also, in this case the
average connection of $\nabla$ and $\nabla^\ast$ is the Levi-Civita
connection of $g$.

The triple $(M,g,\nabla)$ is called a \emph{statistical manifold} if
$\nabla$ is torsion free and $\nabla g$ is completely symmetric.
Equivalently, a statistical manifold is a pseudo Riemannian manifold
$(M,g)$ equipped with a pair of symmetric conjugate connections.
Then, $(g,\nabla,\nabla^\ast)$ is called a \emph{statistical
structure} on $M$. Hence, a statistical manifolds is a
generalization of a pseudo Riemannian manifold.

An almost Norden manifold $(M,J,g)$ equipped with a statistical
structure $(g,\nabla,\nabla^\ast)$ will be called \emph{a
statistical almost Norden manifold}.

\subsection{Conjugate connections with respect to an almost complex structure}

Let $(M,g)$ be a pseudo Riemannian manifold, and $J$ be an almost
complex structure on $M$. If $\nabla$ is an arbitrary linear
connection then the connection $\nabla^\ast$ defined by
\begin{equation}\label{J-conj}
\nabla^\ast_X Y = - J\nabla_X JY = \nabla_X Y - J(\nabla_X J)Y
\end{equation}
is called the \emph{complex conjugate} \emph{connection} of $\nabla$
\cite{Al}, \cite{B}. From (\ref{J-conj}) it follows that
$(\nabla^\ast)^\ast = \nabla$, i.e. $\nabla$ and $\nabla^\ast$ are
mutually conjugate relative to $J$.

A connection $\nabla$ is self-conjugate with respect to $J$ if and
only if it is an \emph{almost complex connection}
($J$-\emph{compatible connection}), i.e. $\nabla J = 0$.

The average connection $\overline{\nabla}=\frac{1}{2}(\nabla +
\nabla^\ast)=\nabla - \frac{1}{2}J\nabla J$ of two complex conjugate
connections is $J$-compatible \cite{Al}.

By the same manner as in \cite{B}, we prove that if $g$ is a Norden
metric then $(\nabla^\ast_X g)(JY,JZ) = - (\nabla_X g)(Y,Z)$. Thus,
$\nabla^\ast g =0$ iff $\nabla g = 0$.

\section{Relations between conjugate connections\\ on almost Norden manifolds}\label{s2}

Let $(M,J,g)$ be an almost Norden manifold. In this section, we
study relations between the aforementioned types of conjugate
connections on $M$.

First, we study the coincidence of conjugate connections with
respect the pair of Norden metrics. Let us remark that if $\nabla$
and $\nabla^\ast$ are conjugate with respect to a Norden metric
tensor $g$, then by (\ref{g-conj}) it follows that $g((\nabla_X
J)Y,Z)=g((\nabla^\ast_X J)Z,Y)$. Hence, in this case $\nabla J = 0$
iff $\nabla^\ast J = 0$.

\begin{proposition}\label{p1}
Let $\nabla$ and $\nabla^\ast$ be linear connections on an almost
Norden manifold $(M,J,g)$. Then, each two of the following
conditions imply the third one:
\begin{enumerate}
\item[(i)] $\nabla$ and $\nabla^\ast$ are conjugate relative to $g$;
\item[(ii)] $\nabla$ and $\nabla^\ast$ are conjugate relative to $\widetilde{g}$;
\item[(iii)] $\nabla J = 0$ \emph{(}$\nabla^\ast J = 0$\emph{)}.
\end{enumerate}
\end{proposition}
\begin{proof}
Let us prove that conditions (i) and (ii) imply (iii). First, we
take into account that $\nabla$ and $\nabla^\ast$ are conjugate with
respect to $g$ and substitute $Y \rightarrow JY$ in (\ref{g-conj}).
Hence, by covariant differentiation and the definition of
$\widetilde{g}$, we obtain
\begin{equation}\label{p1-1}
X\widetilde{g}(Y,Z) = \widetilde{g}(\nabla_X Y, Z) +
\widetilde{g}(Y,\nabla^\ast_X Z) + g((\nabla_X J)Y,Z).
\end{equation}
Then, keeping in mind that $\nabla$ and $\nabla^\ast$ are also
conjugate with respect to $\widetilde{g}$, equality (\ref{p1-1})
implies $\nabla J = 0$.

The truthfulness of the other two statements is proved analogously.
\end{proof}

Proposition \ref{p1} yields the following
\begin{corollary}\label{c1}
Let $(M,J,g,\nabla,\nabla^\ast)$ be a statistical almost Norden
manifold. Then, $(M,J,\widetilde{g},\nabla,\nabla^\ast)$ is also a
statistical almost Norden manifold if and only if $\nabla J = 0$
($\nabla^\ast J = 0$).
\end{corollary}

Let us remark that if $(M,J,g,\nabla,\nabla^\ast)$ and
$(M,J,\widetilde{g},\nabla,\nabla^\ast)$ are simultaneously
statistical manifolds, the Levi-Civita connections $\nabla^0$ and
$\widetilde{\nabla}^0$ of $g$ and $\widetilde{g}$, respectively,
coincide with the average connection of $\nabla$ and $\nabla^\ast$
and hence $\nabla^0 J = \widetilde{\nabla}^0 J = 0$. The last
implies that $(M,J,g)$ and $(M,J,\widetilde{g})$ are both K\"ahler
Norden manifolds.

Next, we study the coincidence of conjugate connections relative to
the metric and the almost complex structure. In this regard, we
prove the following
\begin{proposition}\label{p2}
Let $(M,J,g)$ be an almost Norden manifold, and $\nabla$ be a linear
connection on $M$. Then:
\begin{itemize}
\item[(i)] the conjugate connections of $\nabla$ relative to $g$ and to $J$ coincide if and only if $\nabla \widetilde{g} = 0$;
\item[(ii)] the conjugate connections of $\nabla$ relative to $\widetilde{g}$ and to $J$ coincide if and only if $\nabla g = 0$.
\end{itemize}
\end{proposition}
\begin{proof}
Let us prove (i) (the other statement is proved analogously). The
conjugate connections of $\nabla$ relative to $g$ and  to $J$
coincide if and only if the connection $\nabla^\ast$ defined by
(\ref{J-conj}) satisfies condition (\ref{g-conj}). Keeping in mind
the properties of $g$ and $\widetilde{g}$, the last condition is
equivalent to
\begin{equation}\label{p2-1}
Xg(Y,Z) = g(\nabla_X Y ,Z) - g(JY,\nabla_X JZ).
\end{equation}
Then, by substituting $Z \rightarrow JZ$ in (\ref{p2-1}), we obtain
$X\widetilde{g}(Y,Z) = \widetilde{g}(\nabla_X Y,Z) +
\widetilde{g}(Y,\nabla_X Z)$, i.e. $\nabla \widetilde{g} =0$ which
completes the proof.
\end{proof}

It is well-known that the unique linear connection which is
symmetric and metric with respect to a given metric tensor is the
Levi-Civita connection induced by this metric tensor. In light of
the last fact, Proposition \ref{p2} yields
\begin{corollary}\label{c2}
Let $(M,J,g)$ be an almost Norden manifold, and $\nabla$ be a
symmetric connection on $M$. Then:
\begin{itemize}
\item[(i)] the conjugate connections of $\nabla$ relative to $g$ and $J$ coincide if and only if $\nabla$ is the Levi-Civita connection $\widetilde{\nabla}^0$ of $\widetilde{g}$;
\item[(ii)]the conjugate connections of $\nabla$ relative to $\widetilde{g}$ and $J$ coincide if and only if $\nabla$ is the Levi-Civita connection $\nabla^0$ of $g$.
\end{itemize}
\end{corollary}
Thus, the conjugate connection of $\nabla^0$ (resp.
$\widetilde{\nabla}^0$) relative to $\widetilde{g}$ (resp., to $g$)
is its complex conjugate connection.

The case of a completely symmetric connection $\nabla$ is considered
in the following

\begin{corollary}\label{c3}
Let $(M,J,g)$ be an almost Norden manifold, and let $\nabla$ and
$\nabla^\ast$ be linear connections on $M$. Then:
\begin{itemize}
\item[(i)] If $(M,J,g,\nabla,\nabla^\ast)$ is a statistical manifold, and $\nabla^\ast$ is the conjugate connection of $\nabla$ relative to
$J$ then $(M,J,g)$ is a K\"ahler manifold;

\item[(ii)] If $(M,J,\widetilde{g},\nabla,\nabla^\ast)$ is a statistical manifold, and $\nabla^\ast$ is the conjugate connection of $\nabla$
relative to $J$ then $(M,J,g)$ is a K\"ahler manifold.
\end{itemize}
\end{corollary}
\begin{proof}

(i) Since $(M,J,g,\nabla,\nabla^\ast)$ is a statistical manifold,
the average connection of $\nabla$ and $\nabla^\ast$ is $\nabla^0$.
But because it is also the average connection of two complex
conjugate connections, $\nabla^0$ should be an almost complex
connection, i.e. $\nabla^0 J =0$. Hence, $(M,J,g)$ is a K\"ahler
manifold.

(ii) By a similar manner, we deduce that $(M,J,\widetilde{g})$ is a
K\"ahler Norden manifold, i.e. $\widetilde{\nabla}^0 J =
\widetilde{\nabla}^0\widetilde{g} =0$ which implies
$\widetilde{\nabla}^0 g = 0$. Because $\widetilde{\nabla}^0$ is
symmetric, the last equality yields $\widetilde{\nabla}^0 =
\nabla^0$ and hence $(M,J,g)$ is also K\"ahlerian.
\end{proof}

Based on the results in this section, we conclude that a pair of
linear connections $\nabla$ and $\nabla^\ast$ is conjugate with
respect to all three structural tensor $g$, $\widetilde{g}$ and $J$
simultaneously iff  $\nabla g = \nabla \widetilde{g} = \nabla J = 0$
(which implies $\nabla^\ast=\nabla$). Linear connections preserving
the structural tensors of the manifold by covariant differentiation
are called \emph{natural}(\emph{adapted}). Hence, $\nabla$ is such a
connection.

\section{Conjugate Connections of the Levi-Civita
Connections induced by the pair of Norden metrics}\label{s3}

As seen in the previous section (Corollary \ref{c2}), the
Levi-Civita connections induced by the Norden metrics are the unique
symmetric linear connections on an almost Norden manifold for which
the conjugate connections relative to the associated metric tensor
and the almost complex structure coincide. In this section, we study
curvature properties of these connections.

Let us consider the conjugate connection $\nabla^\ast$ of $\nabla^0$
with respect to $\widetilde{g}$ and $J$, i.e. $\nabla^\ast_X Y =
\nabla^0_X Y - J(\nabla^0_X J) Y$. We remark that $\nabla^\ast$ is a
metric connection, i.e. $\nabla^\ast g = 0$.

If by $R^0$ and $R^\ast$ we denote the corresponding curvature
tensors, according to \cite{B}, we have $JR^\ast(X,Y)Z =
R^0(X,Y)JZ$. Hence, the average curvature tensor $P$ of $R^0$ and
$R^\ast$ defined by (\ref{K}) satisfies the property $P(X,Y)JZ =
JP(X,Y)Z$, meaning that $P$ is a K\"ahler curvature tensor. For
(0,4)-type tensors we have
\begin{equation}\label{P}
g(P(X,Y)Z,W) = \frac{1}{2}\{R^0(X,Y,Z,W)-R^0(X,Y,JZ,JW)\}.
\end{equation}

Next, we focus on the average connection of $\nabla^0$ and
$\nabla^\ast$ which we denote by $D$, i.e. $D_X Y = \nabla^0_X Y -
\frac{1}{2}J(\nabla^0_X J)Y$. Since $\nabla^\ast$ is conjugate to
$\nabla^0$ relative to $\widetilde{g}$ and $J$ simultaneously, the
average connection satisfies $D\widetilde{g}=D J =0$ and hence
$Dg=0$, i.e. $D$ is a natural connection. Moreover, it is the
well-known Lichnerowicz first canonical connection \cite{Li}. In
\cite{Teo}, we have obtained the form of the curvature tensor $K$ of
$D$ on an almost Norden manifold as follows
\begin{equation*}\label{KD}
\begin{array}{l}
g\big(K(X,Y)Z,W\big)=\frac{1}{2}\big\{R^0(X,Y,Z,W) -
R^0(X,Y,JZ,JW)\big\}\medskip\\
+\frac{1}{4}\big\{g\big((\nabla^0_{X}J)Z,(\nabla^0_{Y}J)W\big) -
g\big((\nabla^0_{X}J)W,(\nabla^0_{Y}J)Z\big)\big\}.
\end{array}
\end{equation*}
Then, the last equality and (\ref{P}) yield
\begin{proposition}
On an almost Norden manifold, the average curvature tensor $P$ of
the conjugate connections $\nabla^0$ and $\nabla^\ast$ and the
curvature tensor $K$ of their average connection $D$ are related as
follows
\begin{equation}\label{KP}
\begin{array}{l}
g(K(X,Y)Z,W) = g(P(X,Y)Z,W)\medskip\\ +
\frac{1}{4}\big\{g\big((\nabla^0_{X}J)Z,(\nabla^0_{Y}J)W\big) -
g\big((\nabla^0_{X}J)W,(\nabla^0_{Y}J)Z\big)\big\}.
\end{array}
\end{equation}
\end{proposition}
In \cite{Teo}, we have shown that
$||\nabla^0J||^2=2g^{il}g^{jk}g\big((\nabla^0_{e_{i}}J)e_{k},(\nabla^0_{e_{j}}J)e_{l}
\big)$ on a manifold in the class $\mathcal{W}_1\oplus\mathcal{W}_2$
of the Norden manifolds. Then, if by $\tau(K)$ and $\tau(P)$ we
denote the scalar curvatures of $K$ and $P$, respectively, from
(\ref{theta}) and (\ref{KP}), on a Norden manifold we have
\begin{equation}\label{tKP}
\begin{array}{l}
\tau(K) = \tau(P) + \frac{1}{8}\big(||\nabla^0 J||^2 - 2\
\theta(\Omega)\big).
\end{array}
\end{equation}
In \cite{Teo2}, we have proved that on a manifold in the class
$\mathcal{W}_1$ the relation $\theta(\Omega)= \frac{n}{2}||\nabla^0
J||^2$ is valid. Then, by (\ref{Fcl}) and (\ref{tKP}) we get
\begin{corollary}
On a Norden manifold $(M,J,g)$ belonging to the class
$\mathcal{W}_1$ ($\dim M = 2n \geq 4$) or to $\mathcal{W}_2$ is
isotropic K\"ahlerian iff $\tau(K) = \tau(P)$.
\end{corollary}

Analogous results are valid for the Levi-Civita connection
$\widetilde{\nabla}^0$ of $\widetilde{g}$ and its conjugate
connection $\widetilde{\nabla}^\ast$ relative to $g$ and $J$.

Next, using the characteristic condition (\ref{Fcl}) of the class
$\mathcal{W}_1$, the form (\ref{psi}) of the tensors $\psi_1$ and
$\psi_2$, and by straightforward calculations, we obtain
\begin{proposition}
\noindent Let $(M,J,g)$ be a $\mathcal{W}_1$-manifold. Then, the
curvature tensors $R^\ast$ and $\widetilde{R}^\ast$ of $\nabla^\ast$
and $\widetilde{\nabla}^\ast$, respectively, have the form:
\begin{equation*}
\begin{array}{l}
R^\ast = R^0 - \frac{1}{2n}[\psi_1 + \psi_2](S)
-\frac{\theta(\Omega)}{4n^2}[\pi_1 +\pi_2],\medskip\\
\widetilde{R}^\ast = \widetilde{R}^0 -
\frac{1}{2n}[\psi_1+\psi_2](\widehat{S}) -
\frac{\theta(J\Omega)}{4n^2}[\pi_1+\pi_2],
\end{array}
\end{equation*}
where $\widetilde{R}^0$ is the curvature tensor of
$\widetilde{\nabla}^0$, $S(X,Y) = (\nabla^0_X\theta)JY
+\frac{1}{2n}\theta(X)\theta(Y)$ and $\widehat{S}(X,Y) = -S(X,JY)$.
\end{proposition}
We remark that both $R^\ast$ and $\widetilde{R}^\ast$ are not
(0,4)-type curvature-like tensors.

\section{Statistical structures on almost Norden manifolds}\label{s4}

In this section, we consider statistical structures on almost Norden
manifolds by constructing and studying families of completely
symmetric linear connections.

Let $\nabla$ be a symmetric linear connection, and $Q(X,Y)$ be its
difference tensor with respect to the Levi-Civita connection
$\nabla^0$ of $g$, i.e.
\begin{equation}\label{nabla-0}
\nabla_X Y  = \nabla^0_X Y + Q(X,Y).
\end{equation}
Denote $Q(X,Y,Z)=g(Q(X,Y),Z)$. Then by covariant differentiation we
obtain $(\nabla_X g)(Y,Z) = - Q(X,Y,Z) - Q(X,Z,Y)$. If
$(g,\nabla,\nabla^\ast)$ is a statistical structure, the last
equality and (\ref{stat}) imply that the tensor $Q(X,Y,Z)$ is
completely symmetric, i.e. $Q(X,Y,Z)=Q(Y,X,Z)=Q(X,Z,Y)$, and $\nabla
g = -2Q$. In this case, the connection $\nabla$ is said to be
\emph{completely symmetric}.

By (\ref{g-conj}) and (\ref{nabla-0}) we have
\begin{equation}\label{nabla-star}
\nabla^\ast_X Y = \nabla^0_X Y - Q(X,Y).
\end{equation}

Let us remark that in the theory of statistical manifolds the
(0,3)-type tensor $C(X,Y,Z)=g(\nabla^\ast_X Y - \nabla_X Y, Z) =
(\nabla_X g)(Y,Z)$, which differs from $Q$ only by a factor, is
called the \emph{cubic form} (\emph{skewness tensor}) of the
manifold.

It is known that equality (\ref{nabla-0}) and $\nabla^0 g=0$ imply
the following relation between the curvature tensors $R$ and $R^0$
of $\nabla$ and $\nabla^0$, respectively
\begin{equation}\label{R1}
\begin{array}{l}
g(R(X,Y)Z,W) = R^0(X,Y,Z,W) + (\nabla^0_X Q)(Y,Z,W) \medskip\\ -
(\nabla^0_Y Q)(X,Z,W) + Q(X,Q(Y,Z),W) - Q(Y,Q(X,Z),W).
\end{array}
\end{equation}
Analogously, (\ref{nabla-star}) yields
\begin{equation}\label{R2}
\begin{array}{l}
g(R^\ast(X,Y)Z,W) = R^0(X,Y,Z,W) - (\nabla^0_X Q)(Y,Z,W) \medskip\\
+ (\nabla^0_Y Q)(X,Z,W) + Q(X,Q(Y,Z),W) - Q(Y,Q(X,Z),W),
\end{array}
\end{equation}
where $R^\ast$ is the curvature tensor of $\nabla^\ast$. Then, by
(\ref{R1}) and (\ref{R2}) we obtain \cite{SL}
\begin{equation}\label{R3}
\begin{array}{l}
(\nabla^0_X Q)(Y,Z,W) - (\nabla^0_Y Q)(X,Z,W)\medskip\\
=\frac{1}{2}\{g(R(X,Y)Z,W) - g(R^\ast(X,Y)Z,W)\}.
\end{array}
\end{equation}
Also, since $Q$ is completely symmetric, we have
\begin{equation}\label{R4}
Q(X,Q(Y,Z),W)=g(Q(X,W),Q(Y,Z)).
\end{equation}
Let us denote
\begin{equation}\label{R5}
L(X,Y,Z,W) = g(Q(X,W),Q(Y,Z)) - g(Q(X,Z),Q(Y,W)).
\end{equation}
Since $L$ satisfies properties (\ref{R0}), $L$ is a curvature-like
tensor.

Taking into account (\ref{R3}), (\ref{R4}), (\ref{R5}) and the form
(\ref{K}) of the average curvature tensor (known as the statistical
curvature tensor \cite{FH}) $P$ of $\nabla$ and $\nabla^\ast$, from
(\ref{R1}) we verify
\begin{proposition}\label{p4}
On a statistical manifold, the statistical curvature tensor $P$ and
the curvature tensor $R^0$ are related as follows
\begin{equation}\label{PRL}
P = R^0 + L .
\end{equation}
\end{proposition}
If $\nabla$ is flat, then $\nabla^\ast$ is also flat which imply
$P=0$. Hence, for a flat statistical manifold $R^0=-L$.

If we consider $P$ as the curvature tensor jointly generated by
$\nabla$ and $\nabla^\ast$ then in the next statement we give a
necessary and sufficient condition for the Weyl tensor to be
invariant under the transformation of the Levi-Civita connection
$\nabla^0$ into the pair of symmetric conjugate connections
$(\nabla,\nabla^\ast)$.
\begin{corollary}\label{c-Weyl}
On a statistical manifold, the Weyl tensors of $P$ and $R^0$
coincide iff $W(L)=0$ where $L$ is given by (\ref{R5}).
\end{corollary}

Let $(M,J,g)$ be an almost Norden manifold, and
$(g,\nabla,\nabla^\ast)$ be a statistical structure on $M$. If we
ask for this structure to be compatible with $J$, i.e. $\nabla J =0$
(which implies $\nabla^\ast J =0$) we immediately obtain $\nabla^0 J
=0$. Hence, almost complex completely symmetric connections exist
only on K\"ahler manifolds. Thus, in order to study wider classes of
statistical almost Norden manifolds we will not aim for
$J$-compatibility.

\subsection{Completely symmetric connections constructed by the metrics and the Lie 1-forms}

According to (\ref{Fcl}), an almost Norden manifold which is not in
the class $\mathcal{W}_2\oplus\mathcal{W}_3$ has non-vanishing Lie
1-forms $\theta$ and $\widetilde{\theta}=\theta\circ J$. Thus, on
such manifolds, the pairs of Lie 1-forms and Norden metrics can be
used to construct difference tensors of completely symmetric linear
connections and thus statistical structures. One such family of
connections is introduced in the next
\begin{proposition}
On an almost Norden manifold
$(M,J,g)\not\in\mathcal{W}_2\oplus\mathcal{W}_3$, there exists a
four-parametric family of completely symmetric connections $\nabla$
defined by (\ref{nabla-0}) with difference tensor $Q$ given by
\begin{equation}\label{Q1}
\begin{array}{l}
Q(X,Y) = \lambda_1[\theta(X)Y+\theta(Y)X+g(X,Y)\Omega]\medskip\\
 \phantom{Q(X,Y)}+
 \lambda_2[\theta(JX)Y+\theta(JY)X+g(X,Y)J\Omega]\medskip\\
 \phantom{Q(X,Y)}+\lambda_3
 [\theta(X)JY+\theta(Y)JX+g(X,JY)\Omega]\medskip\\
 \phantom{Q(X,Y)} +\lambda_4
 [\theta(JX)JY+\theta(JY)JX+g(X,JY)J\Omega],
\end{array}
\end{equation}
$\lambda_i \in \mathbb{R}$ ($i=1,2,3,4$).
\end{proposition}
By (\ref{psi}), (\ref{R5}), (\ref{Q1}) and straightforward
calculations we obtain
\begin{proposition}
Let $(M,J,g,\nabla,\nabla^\ast)$ be the statistical almost Norden
manifold with $\nabla$ defined by (\ref{nabla-0}) and (\ref{Q1}).
Then, the statistical curvature tensor $P$ of the manifold has the
form (\ref{PRL}) where
\begin{equation}\label{L1}
\begin{array}{l}
L = \psi_1(S_1) + \psi_2(S_2) \medskip\\
\phantom{L} + [(\lambda_1^2 - \lambda_2^2)\theta(\Omega) +
2\lambda_1\lambda_2\theta(J\Omega)]\pi_1 \medskip\\
\phantom{L} + [(\lambda_3^2 -
\lambda_4^2)\theta(\Omega)+2\lambda_3\lambda_4\theta(J\Omega)]\pi_2\medskip\\
\phantom{L} - [(\lambda_1\lambda_3
-\lambda_2\lambda_4)\theta(\Omega)+(\lambda_1\lambda_4+\lambda_2\lambda_3)\theta(J\Omega)]\pi_3,
\end{array}
\end{equation}
and
\begin{equation*}
\begin{array}{l}
S_1(X,Y) = (\lambda_1^2 + \lambda_3^2 -
2\lambda_2\lambda_3)\theta(X)\theta(Y) + (\lambda_2^2 + \lambda_4^2
+ 2\lambda_1\lambda_4)\theta(JX)\theta(JY)\medskip\\
\phantom{P(X,Y)}+(\lambda_1(\lambda_2+\lambda_3)+\lambda_4(\lambda_3-\lambda_2))[\theta(X)\theta(Y)+\theta(JX)\theta(JY)],\bigskip\\
S_2(X,Y) = (\lambda_3^2-\lambda_4^2)[\theta(X)\theta(Y) -
\theta(JX)\theta(JY)]\medskip\\
\phantom{S(X,Y)}-2\lambda_3\lambda_4[\theta(X)\theta(Y)+\theta(JX)\theta(JY)].
\end{array}
\end{equation*}
\end{proposition}

Since for the Weyl of $\psi_1(S)$ it is valid $W(\psi_1(S))=0$, by
Corollary \ref{c-Weyl}, equalities (\ref{psi}), (\ref{Weyl}) and
(\ref{L1}) we get the following
\begin{proposition}
\noindent Let $\nabla$ be the family of linear connections defined
by (\ref{nabla-0}) and (\ref{Q1}) with the condition
$\lambda_3=\lambda_4=0$. Then, the Weyl tensors of $P$ and $R^0$
coincide.
\end{proposition}

\subsection{Completely symmetric connections constructed by the Lie 1-forms}
A family of completely symmetric linear connections with difference
tensor depending only on the Lie 1-forms $\theta$ and
$\widetilde{\theta}=\theta\circ J$ is presented in the following

\begin{proposition}
On an almost Norden manifold
$(M,J,g)\not\in\mathcal{W}_2\oplus\mathcal{W}_3$, there exists a
four-parametric family of completely symmetric connections $\nabla$
defined by (\ref{nabla-0}) with difference tensor $Q$ given by
\begin{equation}\label{Q2}
\begin{array}{l}
Q(X,Y) = \lambda_1\theta(X)\theta(Y)\Omega + \lambda_2
\theta(JX)\theta(JY)J\Omega \medskip\\
\phantom{Q(X,Y)}+\lambda_3[\theta(X)\theta(Y)J\Omega
+\theta(X)\theta(JY)\Omega + \theta(JX)\theta(Y)\Omega] \medskip\\
\phantom{Q(X,Y)} + \lambda_4[\theta(JX)\theta(Y)J\Omega +
\theta(JX)\theta(JY)\Omega+\theta(X)\theta(JY)J\Omega],
\end{array}
\end{equation}
$\lambda_i \in \mathbb{R}$ ($i=1,2,3,4$).
\end{proposition}

By (\ref{R5}), (\ref{Q2}) and straightforward calculations we obtain
\begin{proposition}
Let $(M,J,g,\nabla,\nabla^\ast)$ be the statistical almost Norden
manifold with $\nabla$ defined by (\ref{nabla-0}) and (\ref{Q2}).
Then, the statistical curvature tensor $P$ of the manifold has the
form (\ref{PRL}) where
\begin{equation*}
L(X,Y,Z,W) = \alpha
[\theta(X)\theta(JY)-\theta(JX)\theta(Y)][\theta(Z)\theta(JW)-\theta(JZ)\theta(W)],
\end{equation*}
where $\alpha =
[\lambda_3^2-\lambda_4^2-\lambda_1\lambda_4+\lambda_2\lambda_3]\theta(\Omega)-
(\lambda_1\lambda_2+\lambda_3\lambda_4)\theta(J\Omega)$.
\end{proposition}
A direct consequence of the last statement and (\ref{PRL}) is that
on manifolds with isotropic Lie vector field $\Omega$ with respect
to both $g$ and $\widetilde{g}$, i.e. satisfying
$\theta(\Omega)=\theta(J\Omega)=0$, we obtain $L=0$, and thus the
statistical curvature tensor $P$ of the statistical structure
defined by (\ref{nabla-0}) and (\ref{Q2}) coincides with the
curvature tensor $R^0$ of $\nabla^0$.

\textbf{Acknowledgement.} The author would like to express her
gratitude to Professor Dr. C.~Udri\c{s}te for his suggestion on the
topic of this paper.

\noindent Marta Teofilova\\ Faculty of Mathematics and
Informatics\\
University of Plovdiv
Paisii Hilendarski\\ 24 Tzar Asen, 4000 Plovdiv\\
marta@uni-plovdiv.bg

\end{document}